\documentclass[oneside,12pt]{amsart}
\usepackage{amsthm,amsmath,amsfonts,amssymb,color, enumerate,tikz, hyperref, algpseudocode, amscd, url, graphicx, floatrow, subfig, float, caption}
\usepackage[all]{xy}

\newtheorem{Theorem}{Theorem}[section]
\theoremstyle{definition}
\newtheorem{Definition}[Theorem]{Definition}
\newtheorem{Notation}[Theorem]{Notation}
\newtheorem{Lemma}[Theorem]{Lemma}
\newtheorem{Corollary}[Theorem]{Corollary}

\newtheorem{Remark}[Theorem]{Remark}

\newtheorem{Proposition}[Theorem]{Proposition}
\newtheorem{Example}[Theorem]{Example}

\CompileMatrices

\numberwithin{equation}{section}

\begin{document}

\title{Non-detectable Patterns Hidden within Sequences of Bits}

\author[D. Allen]{David Allen}
\address{Department of Mathematics BMCC, CUNY, New York, New York 10007}
\email{dtallen@bmcc.cuny.edu}

\author[J. J. La Luz]{Jos\'{e} J. La Luz}
\address{Departmento de Matem\'aticas, Universidad de Puerto Rico,
\newline Industrial Minillas 170 Car 174, Bayam\'on, PR, 00959-1919}
\email{jose.laluz1@upr.edu }

\author[G. Salivia]{Guarionex Salivia}
\address{Department of Mathematics, Computer Science and Statistics, \newline Gustavus Adolphus College, 800 West College Avenue
Saint Peter, MN 56082}
\email{gsalivia@gustavus.edu}

\author[J. Hardwick]{Jonathan Hardwick}
\address{Department of Computer Science, Minnesota state University, \newline Mankato,
South Rd and Ellis Ave, Mankato, MN 56001}
\email{jonathan.hardwick@gmail.com}


\begin{abstract}
In this paper we construct families of bit sequences using combinatorial methods.  Each sequence is derived by converting a collection of numbers encoding certain combinatorial numerics from objects exhibiting symmetry in various dimensions.  Using the algorithms first described in \cite{ALDH} we show that the NIST testing suite described in publication 800-22 does not detect these symmetries hidden within these sequences.

\end{abstract}

\maketitle

\section{Introduction}

In previous work \cite{ALDH} the authors took a random graph $R_G$ and grew it using the cone (Definition \ref{Cone}) to construct families of higher dimensional objects called simplicial complexes.  Each of these complexes comes equipped with a vector that encodes the number of $i$ simplicies.  One of the main thrusts of \cite{ALDH} was to show that the bit sequence derived from the \textit{$f$-vector} of these higher dimensional objects was deemed ``random" by the NIST testing suite \cite{NIST}.  In the current paper, we show that the NIST test suite fails to see patterns in bit sequences derived from combinatorial objects that are symmetrical in a certain sense (Definition \ref{symmetrical}).

Given a simplicial complex $\mathcal{K}$ it is possible to determine a related combinatorial object called a \textit{simple convex polyhedron} whose faces are roughly built out of cones over the geometric realization of certain posets.  In the case that $\mathcal{K}$ is nice, say a simplicial sphere, then the dual $P$ is a \textit{simple convex polytope}.  Polyhedra have $f$-vectors and they encode the number of faces in a given dimension, but we are primarily concerned with those $P$ that are dual to those ``nice" $\mathcal{K}$.  Exactly how this duality works is explained in \S 4.  It is worth noting that determining these vectors and classifying the combinatorial type of $P$ are difficult problems  \cite{KP}.

Simple convex polytopes are interesting and show up in various fields of mathematics. They are well behaved in certain ways and there are other vectors associated to them called the $h$-vector and they exhibit a symmetry described by the Dehn-Sommerville relations.  For such polytopes, these vectors also satisfy a collection of inequalities that are described by the well-known $g$-theorem \cite{B,BP1}.

In this paper we focus on a particularly nice and well-understood family of simple convex polytopes; the standard $n$-simplex $\Delta^n$, for some positive integer $n \geq 1$.  The face structure is understood and the corresponding vectors mentioned above can be determined as well as their dual simplicial complexes (and their respective $f$-vectors).  This group of polytopes is sufficient in demonstrating the claims of the paper.  Namely, the bits derived from complexes built out of the duals of $\Delta^n$ are determined to be random by the NIST test suite even though they are not.

More is true, indeed; if we let $\mathcal{K}$ be the dual of $\Delta^n$, then we construct a family of simplicial complexes $C^j(\mathcal{K})$ and show that they are not dual to a simple convex polytope by applying the $g$-theorem.


\section{Main Results}

The contributions of this paper are:\\[.2ex]
\begin{center}
\begin{enumerate}
  \item Construct a family of simplicial complexes that exhibit a symmetry not detected by the NIST test suite. \\[.8ex]
  \item Show that the family of simplicial complexes are not dual to simple polyhedra. \\
\end{enumerate}
\end{center}

Before stating the main results we list the most pertinent definitions.  Let $\mathcal{K}$ be an $n$-dimensional simplicial complex on the set $\{1,\ldots,n-1\}$, then the \textit{cone} on $\mathcal{K}$ is as follows: $C(\mathcal{K}) =   \{   x\cup \{n\}   |  x\in \mathcal{K}  \}  \cup \mathcal{K}$.  Let $j$ be a positive integer, then when the cone is iterated $j$-times, we refer to it as the $j^{th}$-cone on $\mathcal{K}$ and write it as $C^j(\mathcal{K})$.  The dimension of the complex increases as $j$ increases.

To keep track of the number of $i$ simplicies in $\mathcal{K}$ or $C^j(\mathcal{K})$ there is the \textit{$f$-vector} and it is the vector with the number of $i$ simplicies in the $i^{th}$ component.  We represent the $i^{th}$ component of $\overrightarrow{f}(\mathcal{K})$ by $f_{i}(\mathcal{K})$.  The vector is very often written $(f_0,...,f_{n-1})$ where the last component counts the number of maximal dimensional simplicies in $\mathcal{K}$ (as written this vector encodes the number of $i$-simplicies in an $(n-1)$-dimensional simplicial complex).  A simplicial complex $\mathcal{K}$ is \textit{symmetrical} if $f_i(\mathcal{K}) = f_{n-i-1}(\mathcal{K})$ for $0 \leq i \leq n-1$.  When the components of this vector are converted into bit sequences, the NIST test suite classifies these as random. The results of the testing can be found in the following GitHub public repository ~\url{https://GitHub.com/gnexeng/coning-analysis.git}. This is a fork of the repository used in ~\cite{ALDH}. 

When $\mathcal{K}$ is a \textit{simplicial sphere} (meaning its geometric realization is homeomorphic to a sphere), then there is a simple polytope $P$ that is dual to $\mathcal{K}$ and vice-versa.  We briefly recall the construction; the interested reader can refer to \S4 where a more detailed description is given.  The vertices of $\mathcal{K}$ correspond to the codimension one faces of $P$, the edges correspond to the codimension two faces of $P$ so on and so forth. Given this, the polytope $P$ will satisfy the criteria of the $g$-theorem which describes certain symmetry conditions and inequalities of a vector related to $P$ called the $h$-vector as described by \cite{BP1}. It is related to the vector whose components counts the number of faces of $P$ \cite{BP1}.  Unfortunately, this too, is called the $f$-vector, but when $\mathcal{K}$ is dual to $P$ then there is the following relation that relates the two noting that $(h_0, h_1,\ldots,h_n)$ is the $h$-vector of the dual:  $h_0t^n + \cdots+ h_{n-1}t + h_n  = (t-1)^n + f_0(t-1)^{n-1} + \cdots + f_{n-1}$.

In \cite{B}, the $g$-theorem is stated in terms of certain ``$g_i$" that satisfy a variety of inequalities and they too are related to the components $f_i$ in the $f$-vector of $P$.  If a general $P$ does not satisfy the conditions of the $g$-theorem then it is not simple.  If $\mathcal{K}$ is a random graph $R_G$, then there are equations that follow from the $g$-theorem that one can use to show that the simplicial complexes $C^j(R_G)$ are not simplicial spheres (meaning, the dual $P$ is not simple) for $j$ sufficiently large.\\[.5ex]

The main results are:

\begin{Proposition}\label{No g Theorem}
For a graph $G$ such that $f_1(G) \neq 0$, let $C^{j}(G) = \mathcal{K}^{j+1}$. For $j$ sufficiently large, the complexes $\mathcal{K}^{j+1}$ are not dual to a polytope that satisfies the $g$-theorem.
\end{Proposition}

Proposition \ref{No g Theorem} applies to the case $G = R_G$, a random graph.  In fact, Theorem \ref{No g Theorem} provides us with exactly the number of cones that have to be applied before passing through the class of simple polytopes.
\begin{Theorem} \label{N simplex}
Let $n >1$ and suppose $P^n$ is a polytope such that $\overrightarrow{h}(P^n) = (1,\ldots,1)$, then the $n-1$ dimensional dual simplicial complex $K_P$ is symmetrical.  
\end{Theorem}

From a testing perspective the main results are as follows: 

\begin{enumerate}
  \item The components of the $f$-vectors dual to $P$ in Theorem \ref{N simplex} can be found in Pascal's Triangle.  The NIST test suite views the converted bit sequences as random. (see section \S\ref{bitconversions} and \S\ref{results}\\[.5ex] 

  \item There is a family of complexes given by the $j^{th}$-cone on those complexes constructed from Pascal's Triangle that generate bit sequences that NIST classifies as random. 

\end{enumerate}

The paper is set-up as follows.  In \S 3 we provide a brief overview of simplicial complexes and the algorithms used to generate $C^j(\mathcal{K})$. We also discuss the notion of the $f$-vector and, for the convenience of the reader, list a few calculations made in \cite{ALDH}. In \S 4 polyhedra and their duals are discussed along with their $f$ and $h$-vectors. The relation between the $f$ and $h$-vectors is given as well as specific low dimensional examples. For completeness, the $g$-Theorem is listed and a reference is provided for additional details. The main theorems needed in the sequel are proved in this section too.  \S 5 contains a careful analysis of bit conversions.  Here, we discuss truncation methods that are used to ensure that the bit-stream fits into the NIST suite.  \S 6 lists the algorithm to convert $f$-vectors to $h$-vectors and vice versa and \S 7 contains a summary of the results. \S8 is an Appendix and it contains additional commentary regarding the original code \cite{ALDH}, the improvement that was made and how to replicate the new and improved results.

\section{Simplicial Complexes and Iterated Cones}

For the convenience of the reader we list background material from \cite{ALDH} along with a few other additions.
\begin{Definition}
Let $X$ be a non-empty set. A simplicial complex $\mathcal{K}$ on $X$ is a non-empty subset of the power set of $X$ such that if $x\in \mathcal{K}$ and $y\subset x$ then $y\in \mathcal{K}$.
\end{Definition}

In general $X=\{1,2,\ldots,m\}=[m]$ where $m\in \mathbb{N}$. The sets in $\mathcal{K}$ are called simplices and the dimension of a simplex $x$, denoted $dim(x)$ is defined to be $|x|-1$. The dimension of a simplicial complex is $max\{dim(x)|x\in \mathcal{K}\}$.  If a simplicial complex $\mathcal{K}$ is $n-1$ dimensional we simply write $\mathcal{K}^{n-1}$.  From the definition it follows that any simple graph can be regarded as a simplicial complex of dimension at most one.  If such a graph has at least one edge, then it is a one dimensional.

\begin{Remark}
Notice that since $\mathcal{K}$ is a non-empty set and since $\emptyset \subseteq x$ for all subsets of $X$, then by definition $\emptyset\in \mathcal{K}$.
\end{Remark}

\begin{Example}\label{E1}
The set $\mathcal{K}=\{\emptyset,\{1\},\{2\},\{3\},\{1,2\}\}$ is a one-dimensional simplicial complex on the set $\{1,2,3\}$.\\
\begin{minipage}{\linewidth}
$$\begin{tikzpicture}
\filldraw [black] (-10,0) node[above left =3pt]{$1$} circle (2.5pt);  
\filldraw [black] (-10,2) node[above left =3pt]{$2$} circle (2.5pt);  
\draw[thick] (-10,0) -- (-10,2);
\filldraw [black] (-8,0) node[above left =3pt]{$3$} circle (2.5pt);
\end{tikzpicture}$$
        \end{minipage}
\end{Example}

\begin{Notation}
For a fixed $m\in \mathbb{N}$ let $\widehat{\Delta}^m=P([m])$ where $P([m])$ is the power set on $[m]$.
\end{Notation}
 For instance, $\widehat{\Delta}^2$ is a \textit{2-simplex}.  Pictorially, it is

\begin{minipage}{\linewidth}
$$\begin{tikzpicture}
\filldraw[thick,gray!30]  (-10,0) -- (-10,2)--(-8,0)--(-10,0);
\draw[thick]  (-10,0) -- (-10,2)--(-8,0)--(-10,0);
\filldraw [black] (-10,0) node[above left =3pt]{$1$} circle (2.5pt);  
\filldraw [black] (-10,2) node[above left =3pt]{$2$} circle (2.5pt);  
\filldraw [black] (-8,0) node[above right =3pt]{$3$} circle (2.5pt);
\end{tikzpicture}$$
        \end{minipage}\\

The following can be found in \cite{BP1}
\begin{Definition}\label{Cone}
Let $\mathcal{K}$ be a simplicial complex on the set $\{1,\ldots,n-1\}$. We define and denote the cone over $\mathcal{K}$ as follows:

 $$C(\mathcal{K}) =  \{   x\cup \{n\} |  x\in \mathcal{K}  \}  \cup \mathcal{K}$$

\end{Definition}

Notice that if $\mathcal{K}$ is an $n$-dimensional complex then $C(\mathcal{K})$ is an $(n+1)$-dimensional complex.  Since $C(\mathcal{K})$ is a simplicial complex it is possible to apply the cone again.

\footnotesize
\begin{minipage}{\linewidth}
$$
\begin{tikzpicture}[scale=.8]
\filldraw[black] (-1,0) node[below=.1pt]{$1$} circle (1pt);
\filldraw[black] (1,0) node[below=.1pt]{$2$} circle (1pt);
\draw (0,-.7) node{\footnotesize $dim(\mathcal{K})=0$\normalsize};
\draw[->,thick] (1.2,.6)--(4,.6);
\draw (2.5,.6) node[below=.1pt]{\text{cone}};
\draw (4.5,0)--(5.5,1)--(6.5,0);
\filldraw[black] (4.5,0) node[below=.1pt]{$1$} circle (1pt);
\filldraw[black] (6.5,0) node[below=.1pt]{$2$} circle (1pt);
\filldraw[black] (5.5,1) node[above=.1pt]{$3$} circle (1pt);
\draw[->,thick] (5.5,-1)--(5.5,-3);
\draw (5.5,-.7) node{\footnotesize $dim(C(\mathcal{K}))=1$\normalsize};
\draw (6,-2) node{\text{cone}};
\filldraw[color=gray!20] (4.5,-4.5)--(5.5,-3.5)--(6.5,-4.5)--(5.5,-5.5)--(4.5,-4.5);
\draw (4.5,-4.5)--(5.5,-3.5)--(6.5,-4.5)--(5.5,-5.5)--(4.5,-4.5);
\draw (5.5,-3.5)--(5.5,-5.5);
\filldraw[black] (4.5,-4.5) node[below=.1pt]{$1$} circle (1pt);
\filldraw[black] (6.5,-4.5) node[below=.1pt]{$2$} circle (1pt);
\filldraw[black] (5.5,-3.5) node[above=.1pt]{$3$} circle (1pt);
\filldraw[black] (5.5,-5.5) node[below=.1pt]{$4$} circle (1pt);
\draw (5.5,-6.2) node{\footnotesize $dim(C^{2}(\mathcal{K}))=2$\normalsize};
\end{tikzpicture}
$$
        \end{minipage}\\[2ex]
\normalsize

We write the \textit{$j^{th}$ cone on $\mathcal{K}$} as $C^j(\mathcal{K})$. We have the following from \cite{ALDH} regarding a  graph $G$.

\begin{Lemma}\label{Dim of n cone}
The dimension of $C^j(G)$ is $j+1$.
\end{Lemma}

In general, counting the number of $i$-simplicies for a given simplicial complex is computationally difficult \cite{KP}.  There is a vector that encodes all of these numerics.  We list the following critical definition:

\begin{Definition}\label{fvector}
The $f$-vector of an $(n-1)$-dimensional simplicial complex $\mathcal{K}$, denoted by $\overrightarrow{f}(\mathcal{K})$, is the vector with the number of $i$ simplicies in the $i^{th}$ component. We represent the $i^{th}$ component of $\overrightarrow{f}(\mathcal{K})$ by $f_{i}(\mathcal{K})$.
\end{Definition}

For a given fixed $\mathcal{K}$, the following notation is usually used to denote this vector $\overrightarrow{f}(\mathcal{K}) = (f_0,f_1,\ldots,f_{n-1})$. Often it is simply written as $(f_0,f_1,\ldots,f_{n-1})$. For the purposes of making calculations, recall, $f_{-1}(\mathcal{K}) = 1$ \cite{BP1}.  We have the following calculations from \cite{ALDH}; the proofs can be found there.

\begin{Example}
For Example \ref{E1} we have $\overrightarrow{f}(\mathcal{K})=(3,1)$.
\end{Example}

\begin{Lemma} \label{Count on Cone}
For a graph $G$, $f_0(C^j(G)) = f_0(G) + j$ and $f_1(C^j(G)) =  f_1(C^{j-1}(G)) +  f_0(C^{j-1}(G))$
\end{Lemma}

As an immediate Corollary of Lemma \ref{Count on Cone} we re-write $f_1(C^n(R_G))$ by unraveling the recursion.

\begin{Corollary}\label{recursion}
For a graph $G$ the number of edges of $C^j(G$) can be re-written as $f_1(C^j(G)) = f_1(G) + \sum_{k = 0}^ {j-1} f_0(C^k(G))$
\end{Corollary}

If $R_G$ is a random graph, we re-iterate that the $f$-vector of $C^j(R_G)$ will be the basis, after conversion (see \S Bit Conversions for more details), of  a sequence of bits that is to be analyzed by the suite \cite{NIST}.  In previous work \cite{ALDH}, the authors discussed $C^j(R_G)$ and conducted a series of tests using \cite{NIST} verifying that the construction generated families of random bits coming from a random graph.  The implication is that our initial construction preserves this property as new simplicial complexes are ``grown" from the random graph $R_G$.

When reference to a random graph is made, it refers to those graphs described in \cite{ER} and we use the notation $R_G$ to label such.  These graphs have $f$-vectors:  $(f_0, f_1)$ and we assume that $f_1(R_G) \neq 0$.  For the convenience of the reader we list the algorithm implemented in \cite{ALDH} to generate such graphs and to determine the $f$-vector when the cone is applied successively.\\[.5ex]

\begin{center}
\textbf{Random Graph Algorithm}
\end{center}
\begin{algorithmic}
\State Input:  $n$, $p$
\State{Initialize  $M_{ij} = 0 \; \; \forall i, j$} \State {$M_{ij}$ is the $(ij)^{th}$ entry in the  matrix $M$}
\For{ $ i, j =1,\ldots,n$}
\State{ Generate random number $r$} \EndFor
\If{ $r < p$ and $i < j$} \State{ $M_{ij} = 1$} \Else \State {$M_{ij} = 0$} \EndIf
\State{Construct $R_G$ from the matrix $M$}
\end{algorithmic}

In \cite{ALDH} the following algorithm was used to determine the $f$-vector of $C^j(R_G)$:\\[.5ex]

\begin{center}
\textbf{Computing the $f$-vector of the $j^{th}$-cone on $R_G$}
\end{center}

\begin{algorithmic}
\State Input:  Random graph $R_G$
\State{Compute  $\overrightarrow{f}(R_G)$}
\For{ $j =1,\ldots,n$}
\State{Compute  $\overrightarrow{f}(C^{j}(R_G)$} \EndFor
\end{algorithmic}

Observe:  The random graph $R_G$ is generated using the algorithm above and it is fixed, then we compute $\overrightarrow{f}(R_G)$ to initiate the algorithm mentioned above.  Suppose $\overrightarrow{f}(C^{n-1}(R_G)) = (x_0,x_1,\ldots,x_n)$, then $\overrightarrow{f}(C^{n}(R_G)) = (y_0,y_1,\ldots,y_{n+1})$ where for  $ 1\leq i \leq n+1$

\begin{align*}
&y_0     = x_0 + 1  \\
& y_i    = x_{i-1} + x_i  \\
&y_{n+1} = x_n \\
\end{align*}

\section{Dual Polyhedra and Symmetry}
We begin by following the treatment in \cite{B}.  Additional details concerning polytopes and polyhedra can be found there.  For a fixed integer $q >0$ we work in the Euclidean space $\mathbb{R}^q$  and all scalars will be assumed to be real.  Let $x_1,\ldots,x_n$ be points in $\mathbb{R}^q$, then an \textit{affine combination} is the linear combination

$$\sum_{i=1}^n c_ix_i$$

where $c_1 + \cdots + c_n = 1$.  A collection of points is said to be \textit{affinely independent} if $c_1x_1 + \cdots +c_nx_n = 0$ then the $c_i = 0$.  We say a subset $C$ of $\mathbb{R}^q$ is a \textit{convex set} if the following conditions hold for all $x_1, x_2 \in C$ and scalars $c_1$ and $c_2$ satisfying:

\begin{center}
\begin{enumerate}
  \item $c_1x_1 + c_2x_2 \in C$ such that $c_1, c_2 \geq 0$.\\
  \item $c_1 + c_2 = 1$
\end{enumerate}
\end{center}

\textit{Affine combinations} are linear combinations where the scalars satisfy the conditions above.  More specifically, let $x_1,\ldots,x_n$ be points in $\mathbb{R}^q$, then a convex combination is a linear combination $c_1x_1+ \cdots + c_nx_n$ such that:

\begin{center}
\begin{enumerate}
  \item $c_1 + \cdots + c_n =1$. \\
  \item $c_i \geq 0$
\end{enumerate}
\end{center}

We have the following

\begin{Theorem}
A subset $C$ of $\mathbb{R}^d$ is convex if and only if any convex combination of points from $C$ is again in $C$.
\end{Theorem}
\begin{proof}
\cite{B} Pg 11.
\end{proof}


Given a subset $C$ as above, the intersection of all convex sets in $\mathbb{R}^q$ containing it is a convex set denoted by $conv(C)$.  This set is often referred to as the convex hull of $C$. When the set is clear we will simply say the convex hull.

We have the following definition from \cite{B} pg. 44.

\begin{Definition}
A polytope $P$ is the convex hull of a non-empty finite set.
\end{Definition}

Suppose $P = conv(\{ x_1,\ldots,x_n \})$ for points $x_i \in \mathbb{R}^q$, then the dimension of $P$, $dim (P)$ is $k$ if the following conditions hold for a finite number of points $x_1,\ldots,x_k$ from the set $\{x_1,\ldots,x_n\}$.  First, there is a finite number of points $x_1,\ldots,x_{k+1}$ from the set $\{x_1,\ldots,x_n\}$ that are affinely independent.  Second, there is no such $k+2$ affinely independent points that can be chosen from $\{x_1,\ldots,x_n\}$.

A vertex of $P$ is a zero dimensional face.  The set of all vertices of $P$ will be denoted by $V(P)$ and $|V(P)|$ will be the cardinality of this set.  Proper faces $F$ of $P$ are polytopes such that $V(F) = F \bigcap Ver(P)$ \cite{B} pg. 45 Theorem 7.3.  Since they are polytopes they have a dimension given by the above.

$P$ can also be regarded as the intersection of finitely many halfspaces in a certain Euclidean space \cite{B, BP1} and such sets are often referred to as \textit{polyhedral sets}.  When the context is clear we will simply refer to $P$ as a polytope.  Given a polytope and a finite collection of halfspaces defining it, a supporting hyperplane is a hyperplane $\mathcal{H}$ such that $P \bigcap \mathcal{H} \neq \emptyset$ where the polytope is contained in one of the halfspaces determined by the hyperplane \cite{BP1}. In this context, a \textit{face} of $P$ is $\mathcal{H} \bigcap P = F$.

\begin{minipage}{\linewidth}
\small
$$\begin{tikzpicture}
\filldraw[color=gray!20] (-1,0)--(1,0)--(0,1)--(-1,0);
\draw (-1,0)--(1,0)--(0,1)--(-1,0);
\draw (-2,-1)--(1,2);
\draw[very thick] (-1,0)--(0,1);
\filldraw[black] (-1,0) node[left=.1pt]{$e_1$} circle (1pt);
\filldraw[black] (1,0) node[right=.1pt]{$e_2$} circle (1pt);
\filldraw[black] (0,1) node[ right=.1pt]{$e_3$} circle (1pt);
\draw (-2,-1.2) node{$H$} ;
\draw (-1.5,.7) node{$H\cap P=F$} ;
\end{tikzpicture}$$
\normalsize
\end{minipage}\\[2ex]

Other faces of $P$ include:  $P$, vertices and edges, to name a few.  For a given $P$ its boundary $\partial P = \bigcup_{F \subseteq P} F$.

Given $P$, a facet is a face of dimension $n-1$ and they are often called \textit{codimension one faces}. A $k$-dimensional polytope is called \textit{simple} if each zero face can be written as the intersection of exactly $k$ codimension one faces.  To clarify these points consider the following example, let $P$ be the cube:\\

\begin{minipage}{\linewidth}
$$
\begin{tikzpicture}
\draw (0,0)--(2,0)--(2,2)--(0,2)--(0,0);
\draw (.8,.7)--(2.7,.7)--(2.7,2.8)--(.8,2.8)--(.8,.7);
\draw (.8,.7)--(0,0);
\draw (2,0)--(2.7,.7);
\draw (2,2)--(2.7,2.8);
\draw (0,2)--(.8,2.8);
\filldraw[black] (0,0) node[below=.1pt]{$v_6$} circle (1pt);
\filldraw[black] (.8,.7) node[below right=.1pt]{$v_7$} circle (1pt);
\filldraw[black] (2,0) node[below right=.1pt]{$v_5$} circle (1pt);
\filldraw[black] (2.7,.7) node[below right=.1pt]{$v_8$}circle (1pt);
\filldraw[black] (2,2) node[above left=.1pt]{$v_1$}circle (1pt);
\filldraw[black] (0,2) node[above left=.1pt]{$v_2$} circle (1pt);
\filldraw[black] (2.7,2.8) node[above right=.1pt]{$v_4$} circle (1pt);
\filldraw[black] (.8,2.8) node[above left=.1pt]{$v_3$}circle (1pt);
\draw (1.5,4)node{$F_{5}$};
\draw[->](1.5,3.8)--(1.5,2.6);
\draw (4,1.6)node{$F_{3}$};
\draw[->](3.8,1.6)--(2.3,1.6);
\draw (-.7,.2)node{$F_{2}$};
\draw[->](-.5,.3)--(1.3,1.5);
\draw (-1.5,1.6)node{$F_{1}$};
\draw[-](-1.3,1.6)--(0,1.6);
\draw[dashed,->](0,1.6)--(.5,1.6);
\draw (1.5,-1)node{$F_{6}$};
\draw[-](1.5,-.8)--(1.5,0);
\draw[dashed,->](1.5,0)--(1.5,.5);
\draw (2.8,3.5)node{$F_{4}$};
\draw[-](2.8,3.4)--(2.3,2.8);
\draw[dashed,->](2.3,2.8)--(2,2.4);
\end{tikzpicture}$$
        \end{minipage}

The codimension one faces are $F_1,\ldots,F_6$ (these are the faces of the cube).  Each edge is the intersection of exactly two faces and these are the codimension two faces.  Each vertex is the intersection of exactly three faces (codimension one faces) and so these are the codimension three faces.  One could equally consider the dimensions of the faces rather than the codimension. In such a case, the faces $F_1,\ldots,F_6$ would be the two dimensional faces, the edges the one-dimensional faces and the vertices the zero-dimensional faces of the cube.

\begin{Example}

Following \cite{BP1}, let $P = \Delta^n$. Then this polytope is $conv( \{x_1,\ldots,x_{n+1} \} )$ where $x_i \in \mathbb{R}^n$.  This is called the \textit{$n$-dimensional simplex} and the points are not on a common affine hyperplane.  This is not to be confused with the simplicial complex $\widehat{\Delta}^n$ whose simplicies are the sets in the power set of $\{1,\ldots,n\}$.  Given $\Delta^n$, then for $i \leq n$, $\Delta^i$ is a face of $P$ of dimension at most $n$. If the points $x_i$ are the unit vectors $e_i$, then $\Delta^n$ is called the \textit{standard n-simplex} \cite{BP1} pg. 8.  For example, in $\mathbb{R}^3$, the standard 2-simplex is the convex hull of the unit vectors $e_1, e_2$ and $e_3$.  A picture can be found below:\\

\begin{minipage}{\linewidth}
$$\begin{tikzpicture}
\filldraw[color=gray!20] (1,0)--(0,1)--(-.5,-.5)--(1,0);
\draw[->] (0,0)--(1.6,0);
\draw[->] (0,0)--(0,1.6);
\draw[->] (0,0)--(-1,-1);
\draw (1,0)--(0,1)--(-.5,-.5)--(1,0);
\filldraw[black] (1,0) node[below=.1pt]{$e_2$} circle (1pt);
\filldraw[black] (0,1) node[above left=.1pt]{$e_3$} circle (1pt);
\filldraw[black] (-.5,-.5) node[below right=.1pt]{$e_1$} circle (1pt);
\end{tikzpicture}$$
        \end{minipage}
\end{Example}

We now follow the treatment in \cite{DJ}-specifically pages 425 and 430.  The main idea is that given a polytope one may construct a simplicial complex dual to it and vice-versa.  Assume $P$ is an $n$-dimensional simple polytope (defined as above) and let $\Im = \{F_0,\ldots,F_m\}$ be the set of facets.  Let $\mathcal{K}_P$ denote the dual simplicial complex with vertex set $\{F_0,\ldots,F_m\}$ with the requirement:  $\sigma = \{ F_0,\ldots,F_{j}\}$ span a $j+1$ simplex in $\mathcal{K}_P$ if and only if $ F_0 \cap \cdots \cap F_j \neq \emptyset$ in $P$.  We observe that a zero face in $P$ (a vertex) is a codimension $n$ face; hence, it is the intersection of exactly $n$ facets.  By the definition of the dual, such an intersection spans an $n-1$ simplex.  Therefore, the dual complex $\mathcal{K}_P$ is $(n-1)$-dimensional.  Furthermore, it is a simplicial sphere, roughly meaning, it is essentially a combinatorial representation of the sphere $S^{n-1}$.

\begin{Notation}
We fix the notation regarding the dual.  Given a simplicial complex $\mathcal{K}$ we will write the dual as $P$ (or $P_{\mathcal{K}}$ if we need to stress certain properties of this combinatorial object).  When convenient we may just say ``the dual".  If $P$ is given, then the dual, $\mathcal{K}_P$ will be written as $\mathcal{K}$ when there is no room for confusion.  For reasons of convenience certain authors use the notation $\mathcal{K}^*$ to denote the dual, especially when considering $\mathcal{K}^{**}$ (\cite{G}), but we will stick with the conventions mentioned above.  Given $\mathcal{K}$ the notation $P_{\mathcal{K}}$ is used in the paper \cite{DJ}.
\end{Notation}

\begin{Example}\label{Dual Square}
For the square we have \\
\begin{minipage}{\linewidth}
$$
\begin{tikzpicture}
\filldraw[color=gray!20] (0,0)--(0,2)--(2,2)--(2,0)--(0,0);
\draw (0,0)--(0,2)--(2,2)--(2,0)--(0,0);
\filldraw[black] (0,0) node[below=.1pt]{$v_1$} circle (1pt);
\filldraw[black] (0,2) node[above=.1pt]{$v_2$} circle (1pt);
\filldraw[black] (2,2) node[above=.1pt]{$v_3$} circle (1pt);
\filldraw[black] (2,0) node[below=.1pt]{$v_4$} circle (1pt);
\draw (-.3,1) node{$F_1$} ;
\draw (1,2.3) node{$F_4$} ;
\draw (2.3,1) node{$F_3$} ;
\draw (1,-.3) node{$F_2$} ;
\end{tikzpicture}
\begin{tikzpicture}[scale=.8]
\draw (0,-.5)--(-1.5,1)--(0,2.5)--(1.5,1)--(0,-.5);
\filldraw[black] (0,-.5) node[below=.1pt]{$v_{F_2}$} circle (1pt);
\filldraw[black] (-1.5,1) node[left=.1pt]{$v_{F_1}$} circle (1pt);
\filldraw[black] (0,2.5) node[above=.1pt]{$v_{F_4}$} circle (1pt);
\filldraw[black] (1.5,1) node[right=.1pt]{$v_{F_4}$} circle (1pt);
\draw (-.9,-.1) node{$e_{v_1}$} ;
\draw (-1,2) node{$e_{v_4}$} ;
\draw (1,-.1) node{$e_{v_2}$} ;
\draw (1,2) node{$e_{v_2}$} ;
\end{tikzpicture}$$
\end{minipage}\\[2ex]
\end{Example}

We now give a slightly more complicated example

\begin{Example}

Recall, the cube

\begin{minipage}{\linewidth}
$$
\begin{tikzpicture}
\draw (0,0)--(2,0)--(2,2)--(0,2)--(0,0);
\draw (.8,.7)--(2.7,.7)--(2.7,2.8)--(.8,2.8)--(.8,.7);
\draw (.8,.7)--(0,0);
\draw (2,0)--(2.7,.7);
\draw (2,2)--(2.7,2.8);
\draw (0,2)--(.8,2.8);
\filldraw[black] (0,0) node[below=.1pt]{$v_6$} circle (1pt);
\filldraw[black] (.8,.7) node[below right=.1pt]{$v_7$} circle (1pt);
\filldraw[black] (2,0) node[below right=.1pt]{$v_5$} circle (1pt);
\filldraw[black] (2.7,.7) node[below right=.1pt]{$v_8$}circle (1pt);
\filldraw[black] (2,2) node[above left=.1pt]{$v_1$}circle (1pt);
\filldraw[black] (0,2) node[above left=.1pt]{$v_2$} circle (1pt);
\filldraw[black] (2.7,2.8) node[above right=.1pt]{$v_4$} circle (1pt);
\filldraw[black] (.8,2.8) node[above left=.1pt]{$v_3$}circle (1pt);
\draw (1.5,4)node{$F_{5}$};
\draw[->](1.5,3.8)--(1.5,2.6);
\draw (4,1.6)node{$F_{3}$};
\draw[->](3.8,1.6)--(2.3,1.6);
\draw (-.7,.2)node{$F_{2}$};
\draw[->](-.5,.3)--(1.3,1.5);
\draw (-1.5,1.6)node{$F_{1}$};
\draw[-](-1.3,1.6)--(0,1.6);
\draw[dashed,->](0,1.6)--(.5,1.6);
\draw (1.5,-1)node{$F_{6}$};
\draw[-](1.5,-.8)--(1.5,0);
\draw[dashed,->](1.5,0)--(1.5,.5);
\draw (2.8,3.5)node{$F_{4}$};
\draw[-](2.8,3.4)--(2.3,2.8);
\draw[dashed,->](2.3,2.8)--(2,2.4);
\end{tikzpicture}$$
        \end{minipage}

If $F$ is a codimcension one face of the cube (a face of the cube) we let $v_F$ denote the dual vertex in $\mathcal{K}_{P^3}$. Recall, the codimension one faces of the cube dualize to vertices, so each face of the cube dualizes to a vertex $v_{F_i}$.  Each codimension two face in the cube is the intersection of two faces (e.g., $F_1 \bigcap F_6 = E$ where $E$ is an edge in $P$) dualizes to an edge.  In this very specific example $F_1 \bigcap F_6$ dualizes to the edge $\{v_{F_1}, v_{F_6}\}$.  Finally, a vertex in the cube is a codimension three face since it is the intersection of three facets.  Here, the vertex $v_1$ in the cube is the intersection:  $F_2 \bigcap F_3\bigcap F_5$ so it dualizes to $\{v_{F_2}, v_{F_3}, v_{F_5}\} $  The table below lists the simplicies in the dual two-dimensional simplicial complex $\mathcal{K}_{P^3}$.

\begin{center}
\begin{tabular}{|c|c|c|}
\hline
$0$-simplices&$1$-simplices&$2$-simplices\\ \hline
$v_{F_1}$ & $\{v_{F_1}, v_{F_6}\}$   &  $\{v_{F_2}, v_{F_3}, v_{F_5}\}$ \\
$v_{F_2}$ &  $\{v_{F_2}, v_{F_6}\}$  &  $\{v_{F_1}, v_{F_2}, v_{F_5}\}$             \\
$v_{F_3}$ &  $\{v_{F_3}, v_{F_6}\}$  &  $\{v_{F_1}, v_{F_4}, v_{F_5}\}$           \\
$v_{F_4}$ &   $\{v_{F_4}, v_{F_6}\}$ &  $\{v_{F_3}, v_{F_4}, v_{F_5}\}$            \\
$v_{F_5}$ &   $\{v_{F_1}, v_{F_2}\}$ &   $\{v_{F_2}, v_{F_3}, v_{F_6}\}$          \\
$v_{F_6}$ & $\{v_{F_2}, v_{F_3}\}$   &  $\{v_{F_1}, v_{F_2}, v_{F_6}\}$              \\
$v_{F_7}$ & $\{v_{F_3}, v_{F_4}\}$   &  $\{v_{F_1}, v_{F_4}, v_{F_6}\}$               \\
$v_{F_8}$ & $\{v_{F_1}, v_{F_4}\}$   &   $\{v_{F_3}, v_{F_4}, v_{F_6}\}$              \\
          & $\{v_{F_1}, v_{F_5}\}$   &                   \\
          & $\{v_{F_2}, v_{F_5}\}$   &                   \\
          &  $\{v_{F_3}, v_{F_5}\}$  &                   \\
          &  $\{v_{F_4}, v_{F_5}\}$  &                  \\
\hline
\end{tabular}
\end{center}

\vspace{.1in}
\end{Example}

In the case that a simplicial complex is a simplicial sphere, then the construction can made to generate a simple polytope.  Since the zero simplicies in $\mathcal{K}$ are dual to codimension one faces in the dual polytope, the face structure of the polytope is given by working through higher dimensional simplicies of $K$ and carefully fitting together the faces of the polytope using the definition of the simplicial complex along with the process to dualize.

If $K$ is a general simplicial complex then it is possible to construct a simple polyhedral complex \cite{DJ} pg. 428. This involves the notion of a cone on the geometric realization of various posets.

Later we will consider $\mathcal{K} = C^j(R_G)$ (for some positive integer $j$) and then analyze certain inequalities derived from the $g$-theorem.  To do so requires the analysis of the \textit{$h$-vector} (denoted $\overrightarrow{h}$) of the dual and how it is related to the $\overrightarrow{f}$ of the simplicial complex we construct.  Following \cite{DJ} pg. 430 we have:


\begin{Definition}
Let $(f_0,\ldots,f_{n-1})$ be the $f$-vector of an $(n-1)$-dimensional simplicial complex $\mathcal{K}$. Then the $h$-vector of the dual, $(h_0, h_1,\ldots,h_n)$, is defined by the following equation:

$$h_0t^n + \cdots+ h_{n-1}t + h_n  = (t-1)^n + f_0(t-1)^{n-1} + \cdots + f_{n-1}$$

\end{Definition}

Observe that the number of components in the $h$-vector is the dimension of $\mathcal{K}$ plus one.  In the case that $K$ (simplicial sphere) is dual to $P$ as above, then by \cite{DJ}, $f_i(K)$ is the number of codimension $i+1$ faces of $P$ (or $n-i-1$ faces). From \cite{BP1} pg. 12 there are the following two compact equations that relate specific components of the $\overrightarrow{h}$ and $\overrightarrow{f}$.  For $k = 0,...,n$
we have:\\

\begin{center}
$$h_k = \sum_{i = 0}^k (-1)^{k-i}\binom{n-i}{n-k}f_{i-1}$$
\end{center}

and
\begin{center}
$$f_{n-1-k} = \sum_{q = k}^n \binom{q}{k}h_{n-q}$$
\end{center}

\begin{Example}

Let $\mathcal{K}$ have the following $f$-vector  $(4,4)$.  Referring to Example \ref{Dual Square} it is a square.  The claim is that the dual $P$ has corresponding $h$-vector:  $(1,2,1)$.  Specifically, $h_0 =1 = h_2$ and $h_1 =2$.  We have the polynomial $h_0t^2 + h_1t+h_2 = (t-1)^2 + f_0(t-1) + f_1$.  Expanding and equating coefficients produces $h_0 =  1$, $h_1 = f_0 -2$ and $h_2 = f_1 - f_0 +1$ and we obtain the vector above.

\end{Example}

For certain types of simplicial complexes the $h$-vector of the dual exhibits a symmetry exhibited by Dehn-Sommerville relations.   From \cite{BP1}

\begin{Theorem}\label{Dehn}
The $h$-vector of any simple $n$-polytope is symmetric $h_{n-i} = h_i$ for $i = 0,1,\ldots,n$
\end{Theorem}

\begin{Example}
For a positive integer $n > 1$ let  $P = \Delta^n$ then $h_i(P) = 1$ for $i =0,\ldots,n$.  This follows from \cite{BP1} along with a simple verification using Theorem \ref{Dehn}. For each $i < n-1$, the $h$-vectors of the proper faces $ \Delta^i$ and $ \Delta^{i+1}$ have components all consisting of ones and the $h$-vector of $ \Delta^{i+1}$ has one more component then the $h$-vector of $ \Delta^i$.  More specifically, consider $\Delta^2$; it has $h$-vector $(1,1,1)$.  Now $\Delta^1$ is a face and has $h$-vector $(1,1)$.  We state the following interesting fact from \cite{DJ} and observe that for the examples listed here, the following equation holds.

$$h_0 + \cdots+ h_n = |V(P)|$$

\end{Example}

Recall, $R_G$ is a random graph generated using the methods described in \cite{ALDH} and the simplicial complex $\mathcal{K} = C^j(R_G)$ is the $j^{th}$-cone on $R_G$. A fundamental question we seek to answer is the following; given $\mathcal{K} = C^j(R_G)$ and $\overrightarrow{f}(\mathcal{K})$ what type of symmetry does this $f$-vector exhibit? Another interesting formulation of this question is to determine if the dual polytope is simple or not.  Fortunately, the celebrated \textit{$g$-theorem} provides an answer for a fairly general class of polyhedra. Using \cite{B} we note that the $f$-vector of $P$ has components $f_i$ and they equal the number of $i$-faces of $P$.  So using this convention, if $P$ is simple and $n$-dimensional, then the $f$-vector is $(f_0,...,f_{n-1})$ where $f_0$ is the number of vertices of $P$ and $f_{n-1}$ is the number of facets of $P$.

The following material can be found in \cite{B} pgs. 130-131 and we list it here for the convenience of the reader.  Let $g_i(f) := \sum_{j=0}^d (-1)^{i+j}\binom{j}{i}f_j$ for $i = 0,..,d$.  Let $m: = \lfloor \frac{d-1}{2}\rfloor$ and $n: = \lfloor \frac{d}{2} \rfloor$.  The following is referred to as McMullen's Conditions \cite{B}, but other authors refer to it as the $g$-theorem \cite{BP1}.  For additional information regarding condition $(3)$ the interested reader can refer to \cite{B} pg. 130 (specifically, (2)-(4)).

\begin{Theorem}\label{MM condition}
A $d$-tuple $(f_0,...,f_{d-1})$ of positive integers is the $f$-vector of a simple $d$-polytope if and only if the following conditions hold:

\begin{center}
\begin{enumerate}
  \item $g_i(f) = g_{d-i}(f)$ for $i = 0,..,m$.
  \vspace{.1in}
  \item $g_i(f) \leq g_{i+1}(f)$ for $i = 0,..,n-1$.
  \vspace{.1in}
  \item $g_{i+1}(f) - g_{i}(f) \leq (g_i(f) - g_{i-1}(f))^{\langle i\rangle}$ for $i = 1,...,n-1$.
\end{enumerate}
\end{center}

\end{Theorem}

Theorem \ref{MM condition} encodes information regarding upper and lower bounds on the number of faces of simple polytopes.  From \cite{B} pg. 132 one equation that can be derived is the following:

\begin{center}
$ f_0 = \sum_{i=0}^d \binom{i}{0}g_i(f)$
\end{center}

\vspace{.1in}
This can be written as $f_0 = (d-1)f_{d-1} - (d+1)(d-2)$ and to clarify that this equation refers to the polytope $P$ we will write

\begin{center}
$f_0(P) = (d-1)f_{d-1}(P) - (d+1)(d-2)$
\end{center}


\begin{Remark}
It is important to note that the generated random graph is fixed.  Therefore, the terms $f_0(R_G)$ and $f_1(R_G)$ are fixed, but the parameter $j$ can grow and this is the number of times the cone operation is performed.
\end{Remark}

To apply the $g$-theorem it is critical that we understand the dimension of the simplicial complex that results from applying the cone.  We recall that $C^j(R_G)$ is a $(j+1)$-dimensional simplicial complex when $f_1(R_G) \neq 0$.


\begin{Lemma}\label{f vector cone}
Let $\mathcal{K}$ be an $(n-1)$-dimensional simplicial complex with given $\overrightarrow{f}(\mathcal{K}) = (f_0,\ldots,f_{n-1})$.  For a positive integer $j > 0$ the following holds: $f_{n+j -1}(C^j(\mathcal{K})) = f_{n-1}(\mathcal{K})$\\
\end{Lemma}


\begin{proof}
Let $\sigma_{max} = \{v_1,\ldots,v_n\}$ be a top dimensional simplex in $\mathcal{K}$ noting that $dim(\mathcal{K}) = |  \sigma_{max} | -1$.    We observe that $C(\mathcal{K})$ has the effect of adding one vertex to each simplex, but in particular for a $\sigma_{max}^{'} \in C(\mathcal{K})$ we have $\sigma_{max}^{'} = \sigma_{max} \cup \{c_1\}$ where $c_1$ refers to the new vertex added to the simplex, the cone point, and $\sigma_{max}$ is a simplex counted in $f_{n-1}(\mathcal{K})$.  Clearly, $| \sigma_{max}^{'}|  = |\sigma_{max}| + 1 = n+1$ and there are $f_{n-1}$ such simplicies.  Furthermore, it is obvious that the dimension of $C(\mathcal{K})$ is $n$.  Similarly, for $j >1$ we have $dim(C^j(\mathcal{K})) = dim(\mathcal{K}) + j = n + j -1$.  Each top dimensional simplex in the iterated cone $C^j(\mathcal{K})$ are those that are the top dimensional simplicies in $\mathcal{K}$ with $j$ vertices added coming from adjoining $j$ ``cone" points.  These maximal dimensional simplicies have $n+j$ vertices and so the top component of $\overrightarrow{f}(C^j(\mathcal{K}))$ is $f_{n+j-1}$ and such simplicies are enumerated by $f_{n-1}(\mathcal{K})$.
\end{proof}

\begin{Remark}
 Of particular interest is the case $\mathcal{K} = R_G$ such that $f_1(R_G) \neq 0$.  The equation above, then takes the form:$f_{j+1}(C^j(R_G)) = f_1(R_G)$ noting that $f_{j+1}$ is the component of the $f$-vector enumerating the simplicies of maximal cardinality.
\end{Remark}


\begin{Proposition}\label{No g Theorem}
For a graph $G$ such that $f_1(G) \neq 0$, let $C^{j}(G) = \mathcal{K}^{j+1}$. For $j$ sufficiently large, the complexes $\mathcal{K}^{j+1}$ are not dual to a polytope that satisfies the $g$-theorem.
\end{Proposition}

\begin{proof}

Given a fixed graph $G$ let $f_0(G) = s$ and $f_1(G) = t$.  We argue by contradiction.  Suppose for all $j$ we have $\mathcal{K}^{j+1}$ is dual to a $(j+2)$-dimensional polytope $P$ that satisfies the $g$-theorem.  Then the following equation must hold where $d = dim(P)$:

\begin{center}
$f_0(P) = (d-1)f_{d-1}(P) - (d+1)(d-2)$
\end{center}

Since $f_{j+1}$ is the number of codimension one faces of $P$ we obtain, using the dual complex, $f_{j+1}(P) = s+j$.  Plugging all of this into the equation above
gives:

\begin{center}
$\displaystyle \frac{t-s}{s-2} = j$
\end{center}

To obtain the result observe that the left hand side of the equation is fixed and $j$ can be arbitrarily large.

\end{proof}

Proposition \ref{No g Theorem} also implies that one can pass through the class of simple $n$ polytopes. As a practical matter this is crucial in ensuring that the $f$-vectors do not exhibit a symmetry as described by the statements that follow, in particular Definition \ref{symmetrical}.  One goal is to maintain the computational complexity of the resulting $f$-vector \cite{KP}.



Let $P = \Delta^n$ and $\mathcal{K}_P = \mathcal{K}$ be its dual. It is shown that when $\overrightarrow{h}(P) = (1,1,\ldots,1)$, then $\overrightarrow{f}(\mathcal{K})$ satisfies a Dehn-Sommerville type relationship.

Let us recall Pascal's Triangle:

\large
$$
\begin{tikzpicture}
\foreach \n in {0,...,4} {
  \foreach \k in {0,...,\n} {
    \node at (\k-\n/2,-\n) {${\n \choose \k}$};
  }
}
\end{tikzpicture}
$$
\normalsize

We note that the triangle starts from row zero and the entry $\binom{n}{k}$ refers to row $n$ column $k$. The $n^{th}$ row
in the triangle is

$$\bigg \{ \binom{n}{r} \; \bigg | \;  0 \leq r \leq n \bigg \}$$

There is a zeroth row and it contains $\binom{0}{0}$ and for each row there are columns starting at zero.  For example, row two contains: $\binom{2}{0}$, $\binom{2}{1}$ and $\binom{2}{2}$.  Using this notational convention, there are three columns whereby $\binom{2}{0}$ is in column zero for this row.  There are diagonals too and they begin with the zeroth diagonal consisting of $\binom{y}{0}$ for positive integers $y \geq 0$.

For the convenience of the reader we recall the following.  For $k = 0,\ldots,n$
we have:\\

$$h_k = \sum_{i = 0}^k (-1)^{k-i}\binom{n-i}{n-k}f_{i-1}$$

\vspace{.1in}
$$f_{n-1-k} = \sum_{q = k}^n \binom{q}{k}h_{n-q}$$

\begin{Definition}\label{symmetrical}
An $n-1$ dimensional simplicial complex $\mathcal{K}$ is called symmetrical if $f_j(\mathcal{K}) = f_{n-j-1}(\mathcal{K})$ for $0 \leq j \leq n-1$.
\end{Definition}

\begin{Theorem} \label{N simplex}
Let $n >1$ and suppose $P$ is an $n$-dimensional simple convex polytope such that $\overrightarrow{h}(P) = (1,\ldots,1)$, then the $n-1$ dimensional dual simplicial complex $\mathcal{K}$ is symmetrical.  
\end{Theorem}

\begin{proof}
We assume $n$ is fixed.  Since $h_i = 1$ for each $i$, then for $0 \leq k \leq n$ the components of the
$f$-vector of $\mathcal{K}$ can be written as:

$$f_{n-1-k} = \sum_{q = k}^n \binom{q}{k}$$\\

We first deal with the extreme cases for $k$. If $k = n$, then it is clear that $f_{-1} = 1$.  For $k = 0$, then

$$f_{n-1} = \sum_{q = 0}^n \binom{q}{0} = n+1$$\\

By similar calculations we obtain for $k = n-1$:

$$f_0 = \binom{n-1}{n-1} + \binom{n}{n-1}$$\\

Hence, $f_0 = f_{n-1} = n+1$.  In what follows we will use certain summations in Pascal's triangle.  Generally, consider a set $\{0,\ldots,t\}$ of consecutive positive integers.  Fix $s \in \{0,\ldots,t\}$ such that $s \leq t $ then using the $f$-vector calculation above we have the summation:

$$f_{t-1-s} = \binom{s}{s} + \binom{s+1}{s} +\binom{s+2}{s} + \cdots + \binom{t}{s}$$\\

This is summation along a diagonal in Pascal's Triangle, therefore

$$f_{t-1-s} = \binom{t+1}{s+1}$$\\

Consider $f_{n-1-k}$ and the two substitutions for $k$ (observe that $t = n$ and $s = k$ in the formulation above):\\
\begin{enumerate}
  \item When $ k = j$, then $f_{n-1-j} =\displaystyle \binom{n+1}{j+1}$.\\
  \item When $ k= n-j-1$, then $f_{n-1-j} = \displaystyle\binom{n+1}{n-j}$.\\
\end{enumerate}

To complete the proof one must show that $\binom{n+1}{j+1} = \binom{n+1}{n-j}$, but this follows immediately from the properties of $n$ choose $k$.

\end{proof}

We note that the $\overrightarrow{h}$ has a component $h_n$ but the entries of the dual $(n-1)$ dimensional complex are located in the $n+1$ row in Pascal's Triangle.  The following statement follows immediately from Theorem \ref{N simplex} and the tests found in sections \S\ref{results} and \S\ref{appendix}. 

\begin{Remark}
There are families of bit sequences that are derived from symmetrical simplicial complexes and their successive cones that are determined to be random using the NIST test suite described in publication 800-22~\cite{NIST}.
The key takeaway is that one starts with a combinatorial object with a clear pattern (symmetry) and dualizing it produces a simplicial complex that is highly symmetrical, yet the converted bit sequence shows up as random using the NIST test suite. More is true, using the cone it is possible to construct arbitrarily many bit sequences (derived from this symmetrical object) that the test suite deems random.
\end{Remark}

\section{Bit Conversions} \label{bitconversions}

The experimental implementation uses arbitrary-sized integers to represent vector components, then this vector is converted into a sequence of integers.
Since the NIST test suite \cite{NIST} expects its input to be a binary file, we must convert a sequence of integers into the corresponding binary bit-stream.
To do this we, iterate over the vector components to generate a binary representation of each integer.  We then concatenate this binary representation onto the output file.  Care must be taken; it is important to use a bit-wise approach instead of a byte-wise approach.  In a byte-wise approach, the binary representation of each integer will begin on a byte boundary, and since most integers will not completely fill their most-significant byte, there will be up to 7 extraneous leading zero bits.  In practice, this means that the binary representations of each integer should be bit-shifted before concatenation, to avoid these leading zero bits.

To illustrate the importance of this point, our initial implementation in~\cite{ALDH} has a minor bug in the bit-shifting code that is only triggered for certain byte values, and when triggered drops a 0-bit from the output stream.
Despite the relative rarity of this bug being triggered, the NIST test suite did detect its non-random contributions to the output file, as can be seen in Tables 1, 3, 13, and 15 in~\cite{ALDH}, which show a noticeable ``clustering'' effect towards $p$-values of 0. In the current implementation used in this paper this bug has been fixed.

\section{Algorithms}\label{algorithms} 
All our algorithms implemented are polynomial time and numeric in nature, since we only calculate a vector and its dual.

%
\begin{center}
\textbf{Algorithm for cone operation on $f$-vector of a simplicial complex $\mathcal{K}$.}
\end{center}

\vspace{.1in}

\begin{algorithmic}
\State Input:  $\overrightarrow{f}(\mathcal{K})$ \textbf{as} $x_1,\dots,x_n$
\State Compute  $C^{k}(\mathcal{K})$ \textbf{as} $y_1,\dots,y_{n+1}$
\For{$i =1,\dots,n+1$}
\If{$i = 1$}
 \State $y_1 = x_1 + 1$
\ElsIf{$i = n+1$}
 \State $y_{n+1} = x_n$
\Else
 \State $y_i = x_i + x_{i+1}$
\EndIf
\EndFor
\end{algorithmic}

\vspace{.5in}

\begin{center}
\textbf{Algorithm determining the $f$-vector of a $k$-cone on a simplicial complex $\mathcal{K}$.}
\end{center}

\vspace{.1in}
\begin{algorithmic}
\State Input : Simplicial complex $\mathcal{K}$
\State Compute $\overrightarrow{f}(\mathcal{K})$
\For{$k=1,\dots,n$}
\State Compute $\overrightarrow{f}(C^k(\mathcal{K}))$
\EndFor
\end{algorithmic}

\vspace{.5in}

\begin{center}
\textbf{Algorithm to return the dual $h$-vector of the $f$-vector of $\mathcal{K}$ when $\mathcal{K}$ is dual to $P$.}
\end{center}

\vspace{.1in}

\begin{algorithmic}
\State Input: $\overrightarrow{f}(\mathcal{K})$
\State Compute $\overrightarrow{h}(P)$
\For{$k = 0,\dots,n$}
\For{$i = 0,\dots,k$}
\If{$i = 0$}
\State    $h_k = {n \choose n-k}$
\Else
\State    {$h_k = (-1)^{k-i} $ $n-i \choose n-k$ $ f_{i-1}$}
\EndIf
\EndFor
\EndFor
\end{algorithmic}

\vspace{.5in}

\begin{center}
\textbf{Algorithm to return the dual $f$-vector of the $h$-vector of $P$ when $P$ is dual to $\mathcal{K}$.}
\end{center}

\vspace{.1in}

\begin{algorithmic}
\State Input: $\overrightarrow{h}(P)$
\State Compute $\overrightarrow{f}(\mathcal{K})$
\For{$k = 0,\dots,n$}
\For{$i = 0,\dots,k$}
\If{$i = 0$}
\State    $h_k = {n \choose n-k}$
\Else
\State    {$h_k = (-1)^{k-i} $ $n-i \choose n-k$ $ f_{i-1}$}
\EndIf
\EndFor
\EndFor
\end{algorithmic}


\newpage
\section{Results} \label{results}



\sloppy
The algorithms described in the previous section allow us to perform a particular operation on a simplicial complex; hence changing its $f$-vector.  These vectors will increase in size incrementally and they can be computed since very explicitly, although this is a difficult problem in general.  Although it might be possible to start with a large arbitrary vector as our initial input, we have, by our construction a method to generate bit sequences using a combinatorial construction.  The resulting vectors are converted to a stream of bits, which we then use as input to the National Institute of Standards and Technology (NIST) statistical test suite. You can download the suite at the following url:  \url{https://csrc.nist.gov/Projects/Random-Bit-Generation/Documentation-and-Software}

The suite requires a bit-stream input from a random number generator that contains at least 10 megabits of input data. For each bit-stream, the NIST test suite will output a report that shows the $p$-values associated with each of the tests within the suite.

Recall, one theme of the paper is to demonstrate that the NIST test suite cannot detect certain non-random sequences that arise from combinatorial constructions. We conducted three experiments on certain vectors.  They are: the $h$-vector and variants of its length, the number of cones applied to the dual simplicial complex and the bits associated to the components of the $h$-vector. Our first experiment follows these steps:

\begin{enumerate}
\item Step 1:  Create an $h$-vector consisting of all ones coming from the standard $n$-simplex, $\Delta^n$.
\item Step 2:  Compute the $f$-vector of the dual $\mathcal{K}$ corresponding to $\Delta^n$ and its $h$-vector from the previous step.
\item Step 3:  Compute the iterated cone, $C^k(\mathcal{K})$, for $ 0 \leq k \leq q$ and for some $q$ and the resulting $f$-vector for experiment one listed above.
\item Run the resulting bit sequence (bit-stream) from Step 3 through the NIST test suite.
\end{enumerate}


A non-random pattern should be consistently detected by the NIST test suite and shown with every data set having a $p$-value close to 1.0 or 0.0.  This will  appear as a ``clustering'' of points towards the top or bottom of the graph represented as a flatline (See Figure~\ref{fig:random} and ~\ref{fig:nonrandom}). These sparklines represent the trend of all NIST tests for different variations of the input $h$-vectors.

\begin{figure}[!htbp]
\centering
   \subfloat[random]{\includegraphics[width=.4\textwidth]{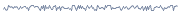}\label{fig:random}}
  \hfill
   \subfloat[non-random]{\includegraphics[width=.4\textwidth]{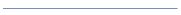}\label{fig:nonrandom}}
  \caption{Sparklines showing the difference between expected NIST results if bit-stream is identified as random (left) vs non-random (right).}
\end{figure}

The time required to run these tests on a modern consumer CPU (Intel Core i5-4590) was originally 2.3 ~\cite{ALDH}. However, performance of the code improved due to optimizations and algorithmic improvements. Updated and improved results from ~\cite{ALDH} can be found in the Appendix~\ref{second-appendix}.

For the first experiment, fix an $h$-vector of length 3751 consisting of all ones. So this is the $h$-vector of $\Delta^{3751}$.  Proceed to steps 2 and 3 and apply the cone to the dual simplicial complex then compute its $f$-vector.  Vary the number of coning operations from 0 to 99 times on the dual complex $\mathcal{K}$ for $0 \leq k \leq 99$ ( figure~\ref{fig:conings}) in Appendix~\ref{first-appendix}.  The sample scatter plots on Figures~\ref{fig:approxentropy21},~\ref{fig:linearcomplexity21} and~\ref{fig:rank21}, show the $p$-value distributions.  The graphs show that the NIST test suite is unable to detect non-randomness.   


\vspace{1cm}

\begin{figure}[h!]
\centering
\caption{Approximate Entropy NIST test $p$-values after applying between 0 - 99 coning operations on the dual of the standard 3751 simplex with h-vector consisting of all 1's }
\includegraphics[width=0.8\linewidth, height=0.5\linewidth]{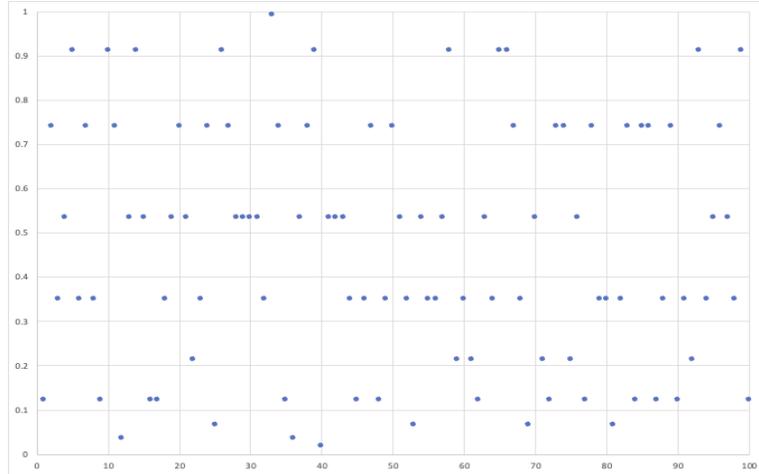}
\label{fig:approxentropy21}
\end{figure}

\begin{figure}[h!]
\centering
\caption{Linear Complexity NIST test $p$-values after applying between 0 - 99 coning operations on a vector of length 3751 with a pattern of all 1's}
\includegraphics[width=0.8\linewidth, height=0.5\linewidth]{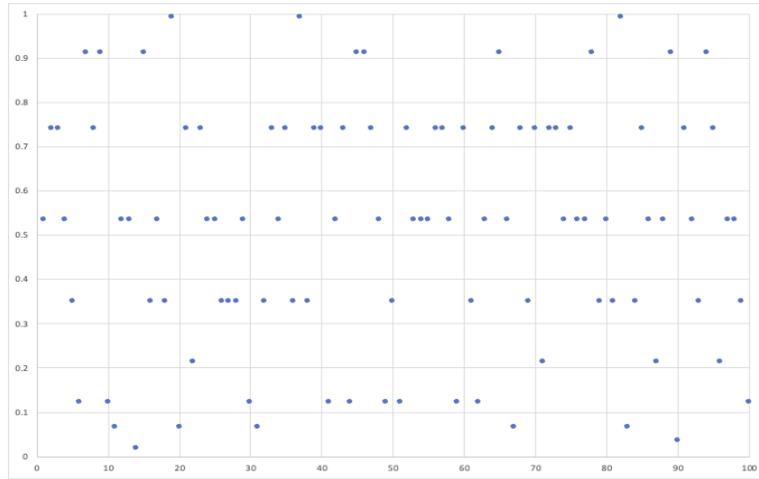}
\label{fig:linearcomplexity21}
\end{figure}

\begin{figure}[h!]
\centering
\caption{Rank NIST test $p$-values after applying between 0 - 99 coning operations on a vector of length 3751 with a pattern of all 1's}
\includegraphics[width=0.8\linewidth, height=0.5\linewidth]{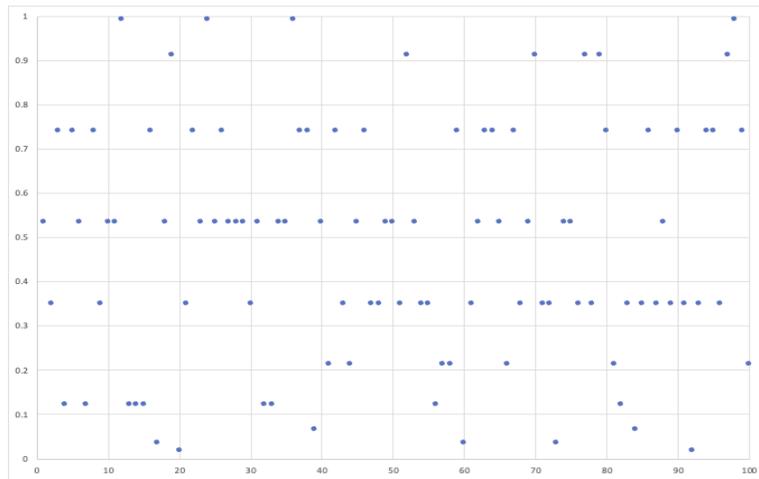}
\label{fig:rank21}
\end{figure}

\clearpage
\indent For the second experiment, we computed the $f$-vector dual to the $h$-vector without applying the cone to the dual (or iterated cone). We vary the $h$-vector length from 3750 to 3849. So, we have a collection of $h$-vectors all of whose components are ones corresponding to a set of standard simplicies:  $\{\Delta^i \;| \; 3750 \leq i \leq 3849\}$.  This should serve as a benchmark since there are no random processes involved that may result from the iterated coning operation. Furthermore, we expect the NIST suite to indicate that the bit-stream is non-random. However, the NIST test suite is unable to detect the non-randomness as illustrated in the sparklines in Figure~\ref{fig:lengths} in Appendix~\ref{first-appendix}.  We have also included sample scatter plots for the Approximate Entropy test (Figure~\ref{fig:approxentropy22}), the Linear Complexity test (Figure~\ref{fig:linearcomplexity22}) and the Rank test (Figure~\ref{fig:rank22}).  

\vspace{1cm}

\begin{figure}[h!]
\centering
\caption{Approximate Entropy NIST test $p$-values on a vectors of length between 3750 and 3849 with a pattern of all 1's, without applying coning operations}
\includegraphics[width=0.8\linewidth, height=0.5\linewidth]{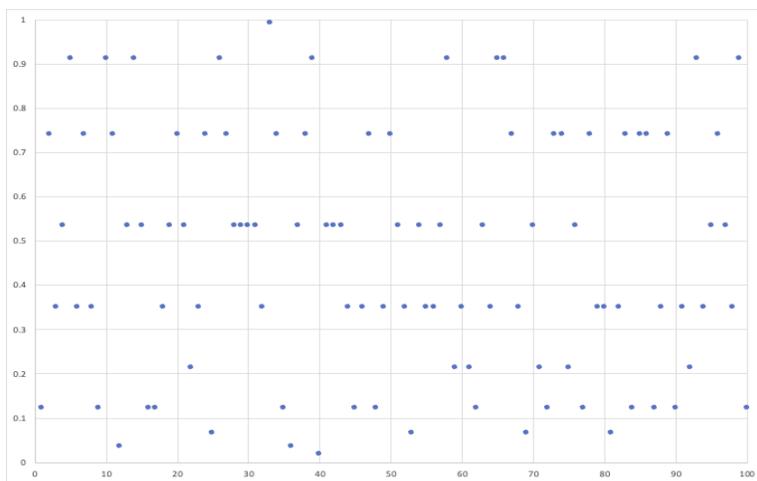}
\label{fig:approxentropy22}
\end{figure}

\begin{figure}[h!]
\centering
\caption{Linear Complexity NIST test $p$-values on a vectors of length between 3750 and 3849 with a pattern of all 1's, without applying coning operations}
\includegraphics[width=0.8\linewidth, height=0.5\linewidth]{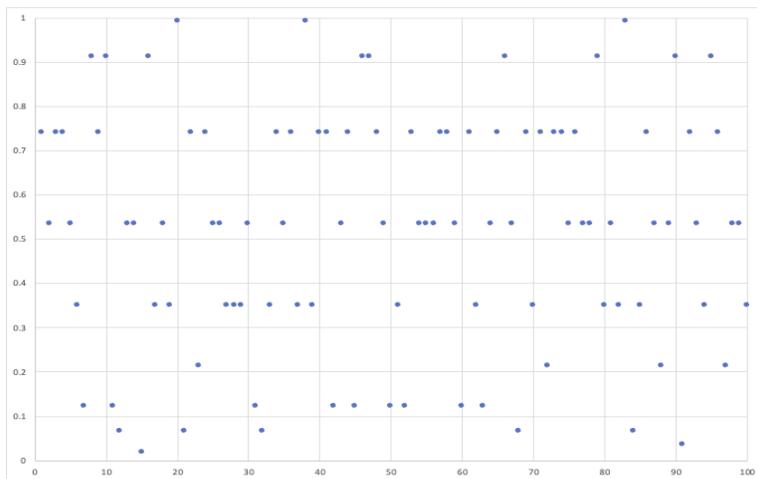}
\label{fig:linearcomplexity22}
\end{figure}

\begin{figure}[h!]
\centering
\caption{Rank NIST test $p$-values on a vectors of length between 3750 and 3849 with a pattern of all 1's, without applying coning operations}
\includegraphics[width=0.8\linewidth, height=0.5\linewidth]{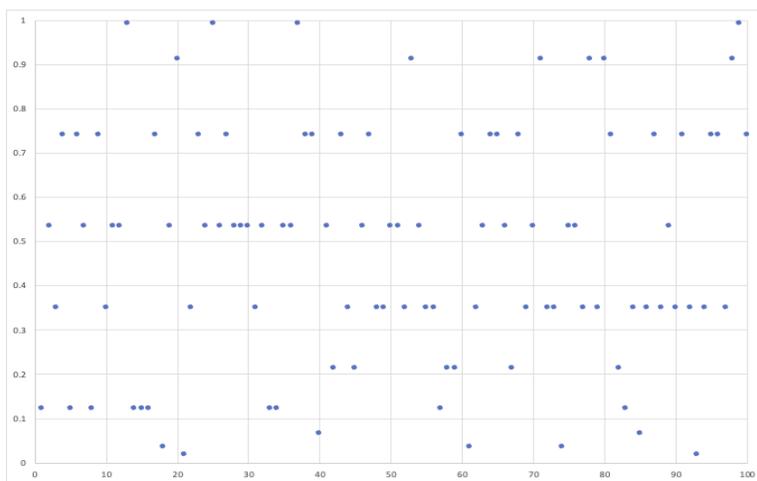}
\label{fig:rank22}
\end{figure}

\clearpage
\indent Motivated by the failure of the NIST test suite to detect non-random phenomena we conducted a third experiment. Some terminology is needed.  Given an $h$-vector $(h_0,...,h_n)$ define \textit{end elements} to be $h_0$ and $h_n$ whilst, \textit{non-end} elements are those $h_i$ such that $i \neq 0,n$.
For $\Delta^{3750}$ fix the corresponding $h$-vector.  Vary the non-end elements in the vector from 1 to 100. This must be done in a manner where the Dehn-Sommerville relations are preserved (see Theorem 4.9) as a necessary, but not sufficient condition for the vector to correspond to a polytope.  We observe that the construction is general and allows for components of the resulting vector to have non-end points to be any integer so long as the two end elements are 1, but for the purposes of testing, we restrict to integers between 1 and 100.  We are not claiming that the resulting vector is the $h$ of a polytope.  Indeed, other condition(s) would have to hold for this to be true.  From a testing perspective, this ad-hoc method allows for additional bit generation.  Moreover, for those vectors that correspond to polytopes, the construction allows for additional vectors to be created by dualizing and computing the $f$-vector.  Some examples of varied non-end elements can be found below:    

\begin{itemize}
\centering
\item[] 1 1 1 \ldots 1 1 1
\item[] 1 2 2 \ldots 2 2 1
\item[] 1 3 3 \ldots 3 3 1
\item[] 1 4 4 \ldots4 4 1
\item[] \ldots
\item[] 1 100 100 \ldots 100 100 1
\end{itemize}


We included the sparklines (see Figure~\ref{fig:patterns} in Appendix~\ref{first-appendix}) and sample scatter plots for the Approximate Entropy test (see Figure~\ref{fig:approxentropy23}), the Linear Complexity test (see Figure~\ref{fig:linearcomplexity23}) and the Rank test (see Figure~\ref{fig:rank23}).  

\vspace{1cm}

\begin{figure}[h!]
\centering
\caption{ApproximateEntropy NIST test $p$-values on vectors of length 3750 with non-end elements from 1 to 100.}
\includegraphics[width=0.8\linewidth, height=0.5\linewidth]{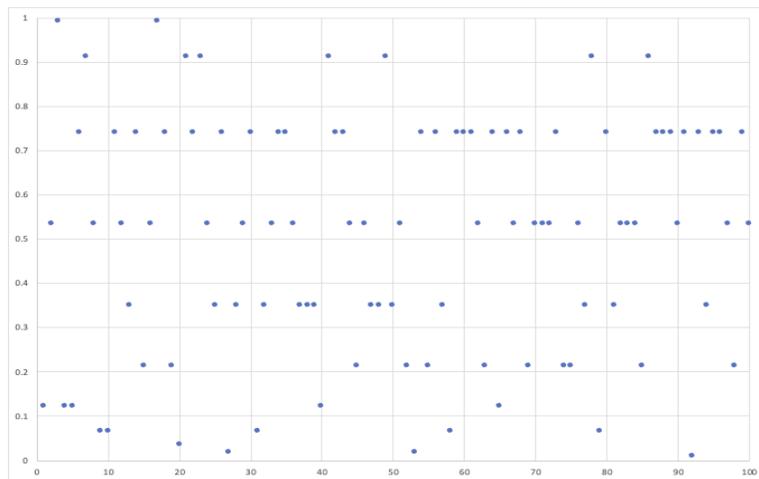}
\label{fig:approxentropy23}
\end{figure}

\begin{figure}[h!]
\centering
\caption{LinearComplexity NIST test $p$-values on vectors of length 3750 with non-end elements from 1 to 100.}
\includegraphics[width=0.8\linewidth, height=0.5\linewidth]{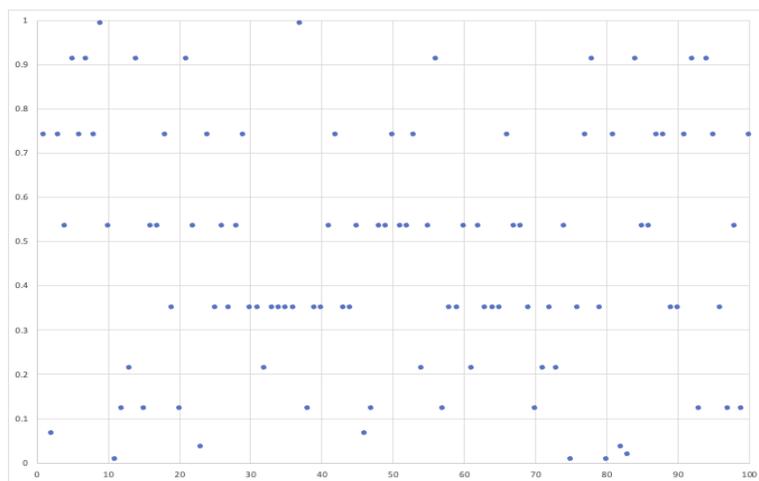}
\label{fig:linearcomplexity23}
\end{figure}

\begin{figure}[h!]
\centering
\caption{Rank NIST test $p$-values on vectors of length 3750 with non-end elements from 1 to 100.}
\includegraphics[width=0.8\linewidth, height=0.5\linewidth]{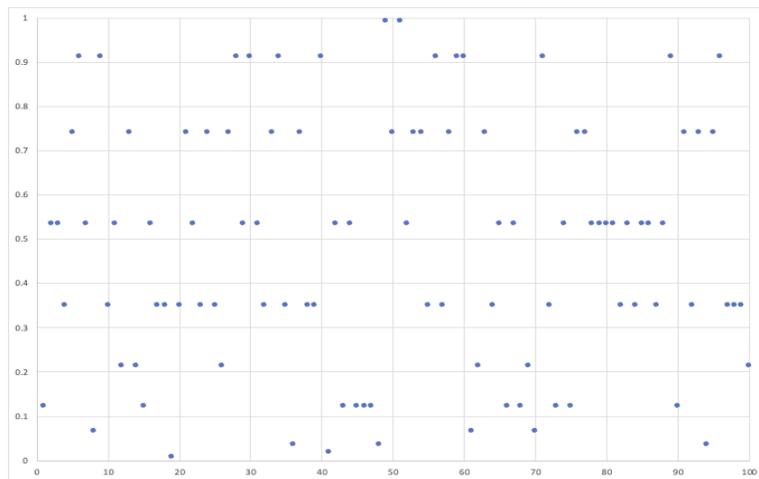}
\label{fig:rank23}
\end{figure}

\clearpage
\appendix

\section{Appendix A} \label{first-appendix}
This appendix includes the sparklines for all three experiments described in the Results section~\ref{results}. Each sparkline represents all p-values for all tests of the NIST suite for a particular bit-stream input.

\begin{figure}[h!]
\centering
\caption{100 trendlines for all NIST tests $p$-values after applying between 0 and 99 coning operations on a vector of length 3751, with a pattern of all 1's}
\includegraphics[width=0.2\linewidth, height=1.2\linewidth]{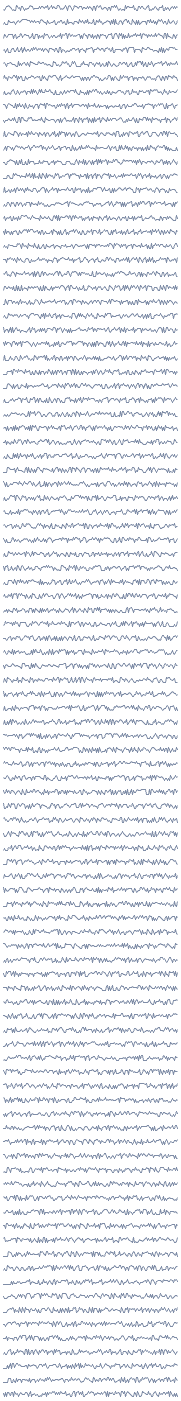}
\label{fig:conings}
\end{figure}

\begin{figure}[h!]
\centering
\caption{100 trendlines for all NIST tests $p$-values on vectors of lengths between 3750 and 3849 with a pattern of all 1's}
\includegraphics[width=0.2\linewidth, height=1.2\linewidth]{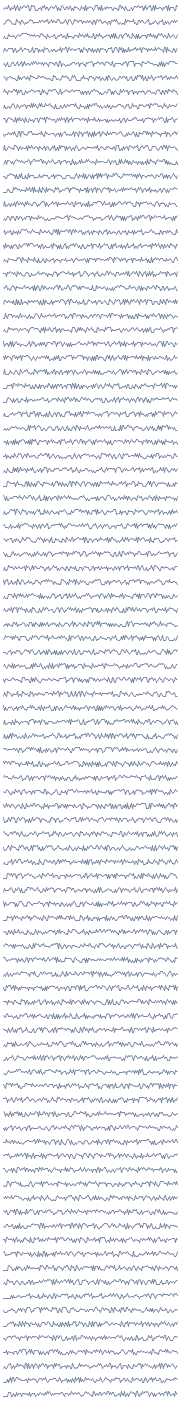}
\label{fig:lengths}
\end{figure}

\begin{figure}[h!]
\centering
\caption{100 trendlines for all NIST tests $p$-values on vectors of length 3750 varying non-end elements from 1 to 100.}
\includegraphics[width=0.2\linewidth, height=1.2\linewidth]{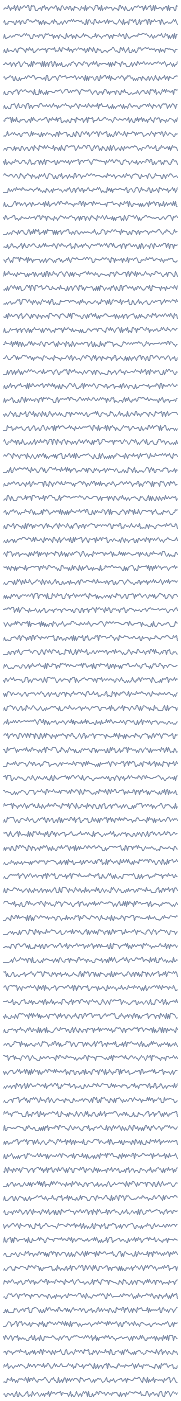}
\label{fig:patterns}
\end{figure}

\clearpage
\section{Appendix B}\label{second-appendix}

Since the appearance of ~\cite{ALDH} the code has been updated.  Version 2.0 of the code fixes several bugs, including one where 0's were not being properly appended to the bit-stream. As a result, fewer false positives appear in the output from the NIST test suite. The updated code was used to run all the tests in this paper and to re-run the experiments from~\cite{ALDH}. The interested reader can refer to the following GitHub repository for the code, the updated results and the results from the experiments conducted in this publication:

\vspace{.5cm}

\begin{center}
\url{https://github.com/gnexeng/coning-analysis}
\end{center}

\vspace{.5cm}

The following is a sample of the updated results from our first publication~\cite{ALDH}

%

\begin{figure}[h!] 
\begin{minipage}[t]{.5\linewidth}
\includegraphics[width=\linewidth]{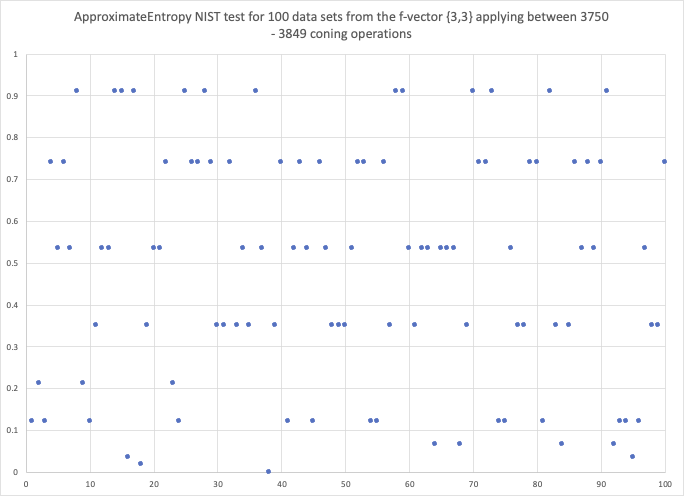}
\end{minipage}\hfill
\begin{minipage}[b]{.5\linewidth}
\includegraphics[width=\linewidth]{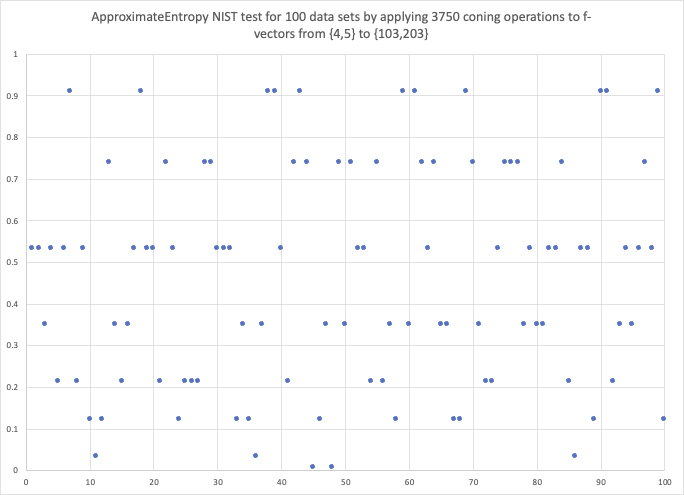}
\end{minipage}
\end{figure}

%

\begin{figure}[h!] 
\begin{minipage}[t]{.5\linewidth}
\includegraphics[width=\linewidth]{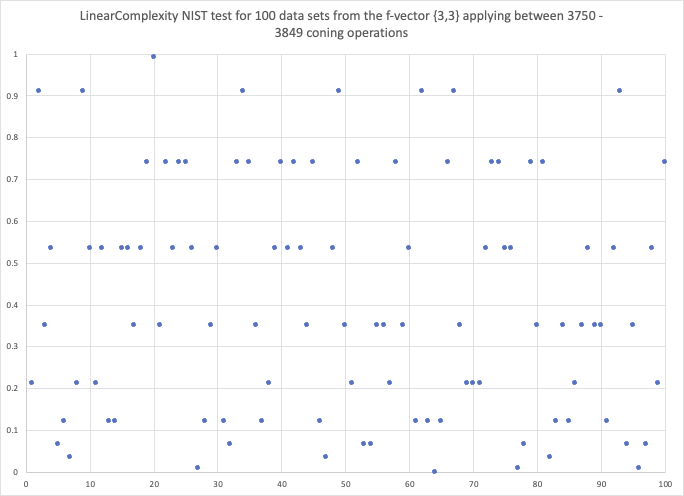}
\end{minipage}\hfill
\begin{minipage}[b]{.5\linewidth}
\includegraphics[width=\linewidth]{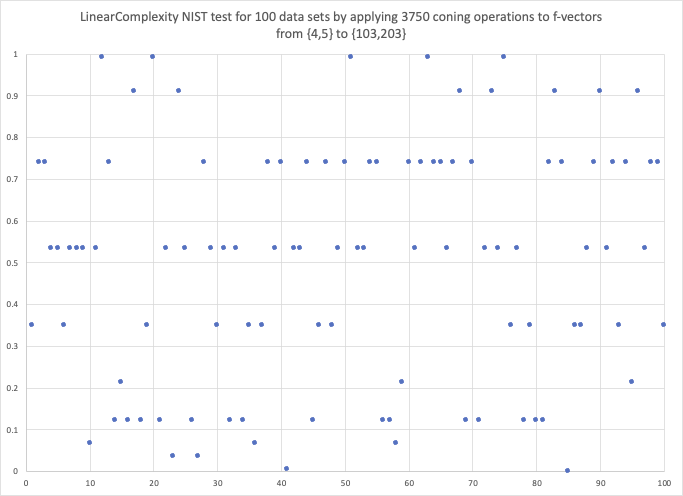}
\end{minipage}
\end{figure}

%
\newpage

\begin{figure}[h!] 
\begin{minipage}[t]{.5\linewidth}
\includegraphics[width=\linewidth]{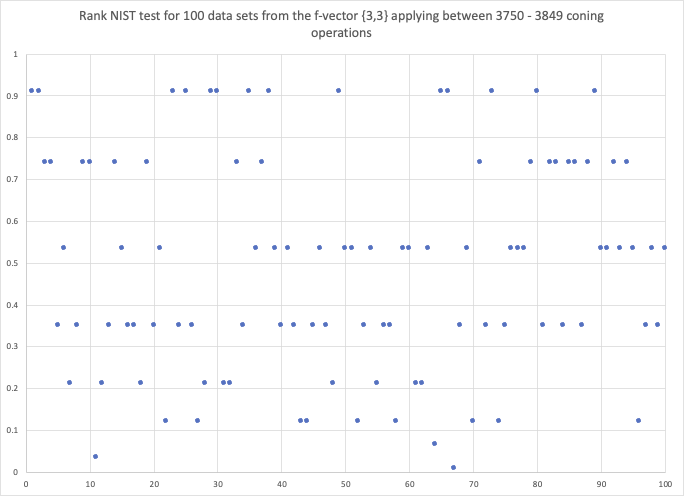}
\end{minipage}\hfill
\begin{minipage}[b]{.5\linewidth}
\includegraphics[width=\linewidth]{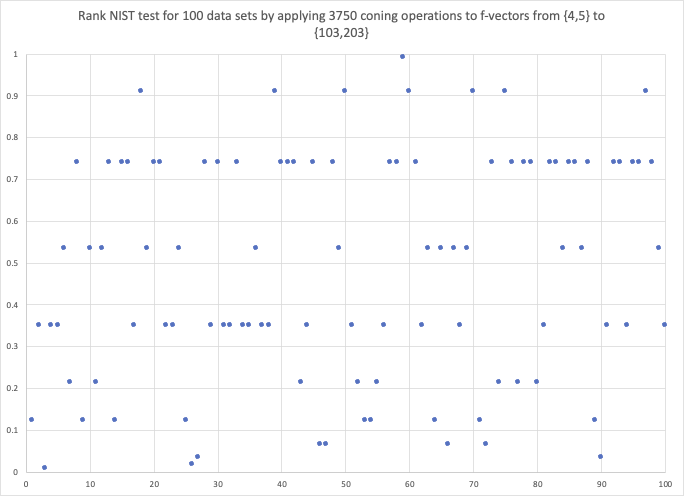}
\end{minipage}
\end{figure}


\end{document}